\documentclass{article}
\usepackage{graphicx} 
\usepackage[margin=2.5cm]{geometry} 
\usepackage{amsthm,amsmath,amssymb}
\usepackage{tikz}
\usepackage[utf8]{inputenc}
\usepackage[T5]{fontenc} 
\usepackage[vietnamese, english]{babel}
\usepackage[linesnumbered,ruled,vlined]{algorithm2e}
\usepackage{float}
\usepackage{enumitem}
\usepackage{subcaption}
\usepackage{array}
\usepackage{url}

\title{On the Number of Zero Forcing Minimal Forts on Trees}

\author{Nguyễn {Hoàng} {Đạt}\thanks{Vietnam National University - Hanoi University of Science} ~and Kenter, Franklin H. J.\footnotemark[1] \thanks{U.S. Naval Academy}}

\date{May 8, 2026}

\newtheorem{lemma}{Lemma}
\newtheorem{theorem}{Theorem}

\newtheorem{definition}{Definition}
\newtheorem{proposition}{Proposition}
\newtheorem{cor}{Corollary}
\newtheorem{question}{Question}

\begin{document}

\maketitle

\newcommand\blfootnote[1]{%
  \begingroup
  \renewcommand\thefootnote{}\footnote{#1}%
  \addtocounter{footnote}{-1}%
  \endgroup
}

\blfootnote{~\\ The views
expressed in this article are those of the authors and do not reflect the official policy or position of the U.S. Naval Academy, the Department of the Navy, the
Department of War, or the U.S. Government.}

\begin{abstract}

We solve a conjecture by Becker et al. (2025) on the topic of zero forcing regarding the number of minimal forts of a tree. They conjectured and we prove $\mathcal{F}_{T_n} \le \binom{n}{2} \mathcal{F}_{P_n}$ where $\mathcal{F}_{T_n}$ is the maximum number of minimal forts on a tree on $n$ vertices and $\mathcal{F}_{P_n}$ is the number of minimal forts of the path graph on $n$ vertices.  Our solution relies on both a computational and theoretical approach. Computationally, we introduce and implement an efficient algorithm to compute the exact number of minimal forts for small trees; this is used to establish the large base case required for our strong induction. Theoretically, we provide an adaptation of the recursion relation that defines $\mathcal{F}_{P_n}$ that applies for all forests; this is used in the induction step to establish the result.

In addition, several examples and new open problems are presented.
\end{abstract}

\section{Introduction}

Zero forcing on graphs was originally conceived in 2008 to bound the maximum nullity of the family of symmetric matrices described by a graph \cite{AIM}. Zero forcing can be thought of as a game or propagation process whereby an initial set of vertices are colored and all others uncolored. Thereafter, a color-change rule is iteratively applied: if a colored vertex has only one uncolored neighbor, it {\it forces} that neighbor to be colored. The color-change rule emulates backsolving a system of equations with arbitrary coefficients. Hence, zero forcing and its variations have roles in applications where the pattern or schematic of interactions are known {\it a priori} but the decision of what or where to take measurements (i.e., the initially colored vertices) must be determined before the exact quantitative details of those interactions are known; among these are power monitoring \cite{brueni2005pmu,liao2015hybrid}, controllability of networks \cite{monshizadeh2014zero}, quantum control \cite{burgarth2013zero}, and leader-follower dynamics \cite{abbas2022leader}.

 While the determination of the zero forcing number, $Z(G)$, is known to be an NP-hard problem \cite{aazami2008hardness,trefois2015zero}, recent attention has shifted toward combinatorial optimization methods in linear programming (see for instance \cite{brimkov2021improved, cameron2023forts, hicks2024many}). In turn, understanding structural obstructions enables practical approaches to compute the zero forcing number and its variations. These obstructions, formally defined as forts in \cite{fast2018effects}, form a sort of ``combinatorial dual'' of zero forcing sets. A nonempty set of vertices is a {\it fort} if every vertex outside of the fort is adjacent to either zero or two or more vertices in the fort. In other words, a fort is the complement of a set of colored vertices whereby no more forces are possible. This enables zero forcing to be considered as a covering problem: determining the zero forcing number is equivalent to finding a minimum size set that intersects every fort in the graph (See \cite{brimkov2021improved, hicks2024many}).
 
The enumerative aspects of forts-- specifically minimal forts-- have  consequential impacts toward this computational approach:  Fewer minimal forts implies fewer constraints which leads to more efficient optimization algorithms for zero forcing \cite{hicks2024many}. The most comprehensive study on enumerating minimal forts in graphs was done by Becker, Cameron, Hanely, Ong, and Previte \cite{becker2025number} who initiated a systematic investigation for various families of graphs and trees and the asymptotic growth rate of the number of minimal forts. In their work, they made two conjectures on the number of minimal forts for trees as well as graphs in general.

In this article, we prove their first conjecture regarding the number of minimal forts on trees. Let $\mathcal{F}_{T_n}$ (respectively $\mathcal{F}_{R_n}$) be the maximum number of minimal forts over all trees (respectively forests) on $n$ vertices. 
Let $\mathcal{F}_{P_n}$ be the number of minimal forts on the path on $n$ vertices. Our main result is the following:

\begin{theorem}[From Conjecture 30 in \cite{becker2025number}] \label{thm:main} For $n\ge 3,$
    \[ \mathcal{F}_{T_n} \le \mathcal{F}_{R_n} \le \binom{n}{2} \mathcal{F}_{P_n} \]
\end{theorem}

This bound posits strict bounds on the asymptotic growth rates for the number of minimal forts in trees: they cannot grow much faster than the path graph, $P_n$.  As detailed in \cite{becker2025number} and as we will recall in Section \ref{sec:prelim}, $\mathcal{F}_{P_n} \approx  k_1 \times 1.32472^n$ for some constant $k_1$. Therefore, Theorem \ref{thm:main} limits the number of minimal forts for any tree to approximately $1.32472^{n+o(n)}$. As it turns out, we will show that for at least $n \le 73$, there are trees on $n$ vertices with more minimal forts than the path graph $P_n$; hence, removing the multiplier factor of ${n \choose 2}$ entirely is not possible. 

Our main approach is largely combinatorial by adapting the recursion relation that defines $\mathcal{F}_{P_n}$ to trees in general (Lemma \ref{lem:recurr}). One obstacle is that the base case for this approach is rather large. While we do not join the small but growing trend of using ``AI" for mathematical proof, our base case for strong induction requires calculating the number of minimal forts for all forests up to $n=19$ vertices. Hence, we resort to a computer and an optimized algorithmic search to establish our base case (Section \ref{sec:algorithm} and Algorithms \ref{alg:minimal_forts} and \ref{alg:decide}).

In summary, our contributions are the following:
\begin{itemize}
    \item We give a purely structural characterization of minimal forts on trees. (Lemma \ref{lem:lemmaforalgorithm}, Section \ref{sec:lemmas})
    \item We develop an algorithm that efficiently computes the number of minimal forts on trees. (Algorithms \ref{alg:minimal_forts} and \ref{alg:decide}, Section \ref{sec:algorithm})
    \item We apply the algorithm to determine the case of trees on $n$ vertices that have the maximum number of minimal forts for $n=4$ to $n=19$ among other computational results. (Sections \ref{sec:basecase} and \ref{sec:compresults})
    \item We provide a general path-like recursion relation that applies to all trees. (Lemma \ref{lem:recurr}, Section \ref{sec:recurr})
    \item We develop combinatorial inequalities and prove the main result, Theorem 1. (Section \ref{sec:mainresult})
    \item We generate several open equations from our study. (Section \ref{sec:conclude})
\end{itemize}

\section{Definitions and Preliminaries} \label{sec:prelim}

\subsection{Graph Theoretic Terminology}

Throughout, we will use standard graph theoretic terminology. For a graph (or subgraph) $G$, we use $V(G)$ to denote the vertex set of $G$. All graphs will be simple and finite, and unless specified, undirected. A {\it forest} is a graph with no cyclic subgraph and a {\it tree} is a connected forest. A {\it leaf} of a graph is a vertex with degree one. Two vertices $u$ and $v$ are {\it adjacent} if they share an edge, which we denote by $u \sim v$, in which case we also say that $u$ is a neighbor of $v$ and vice versa. The neighborhood of $u$ is the set of all neighbors of $u$, not including $u$; we denote the neighborhood as $N(u)$. 

We let $P_n$ denote the path graph on $n$ vertices (which has $n-1$ edges). We let $S_n$ denote the star graph on $n$ vertices.   We define a special tree $T(n,k,m,p)$ to be the rooted tree with height 2 with root $r$ where root $r$ has $k$ children, among those $k$ children $k-p$ have $m$ children and $p$ have $m-1$ children; hence, necessarily, the number of vertices is $n = 1 + k + (k-p)m + p(m-1) = 1 + k + k m - p$.

Within a tree, $T$, the \emph{branch} or \emph{path branch} for a leaf $v$ is the maximal connected induced subgraph that contains $v$ and only vertices of degree one or two in $T$. The \emph{neighbor of a path branch} is the unique vertex (if it exists) not in the path branch that has a neighbor as a vertex in the path branch.

The study of forts arises from zero forcing. Recall that in zero forcing, an initial set of vertices are {\it colored} and others {\it uncolored}. The {\it color-change rule} says that a colored vertex with only one uncolored neighbor may {\it force} that uncolored vertex to become colored. A set of vertices is a {\it zero forcing set} if, when colored and after iteratively applying the color-change rule, all vertices of the graph become colored. We will refer to a set of vertices that is not a zero forcing set as a {\it failed forcing set}.

A {\it fort} of a graph $G$ is a nonempty set of vertices $F \subseteq V(G)$ such that every vertex in $\overline F$ is adjacent to either zero or two or more vertices in $F$. In the context of zero forcing, $F$ is a fort if whenever $\overline F$ is colored, no more forces are possible; hence if $F$ is a fort, $\overline F$ is a failed forcing set. Because of this relationship, we may refer to vertices in $\overline F$ as ``colored''. A fort, $F$, is called {\it minimal} if there is no proper subset of $F$ that is also a fort. 

We have the following property:

\begin{proposition}
A minimal fort $F$ has the property that if $\overline F$ and any one additional vertex are colored, then the entire graph will eventually be colored.
\end{proposition}

\begin{proof}
    Suppose $F \subset G$ is a minimal fort such that there is a vertex $f \in F$ such that $\overline F \cup \{f\}$ will not force the entire graph. Let $E$ be the set of vertices that are colored after exhaustively applying the color-change rule when $\overline F \cup \{f\}$ is initially-colored. Since no more forces are possible, $\overline E$ is a fort. However necessarily, $E \supseteq \overline F \cup \{f\} \supseteq \overline F$ and so $\overline E \subseteq F$ which violates the minimality of $F$.
\end{proof}
For emphasis, a minimal fort need not be a minimum-sized fort.

This study largely concerns enumerating the number of minimal forts in a given graph. For a graph $G$, let $\mathcal{F}_{G}$ denote the number of minimal forts of $G$. In  a mild abuse of notation, we let $\mathcal{F}_{T_n}$ denote the maximum number of minimal forts over all trees on $n$ vertices and we let $\mathcal{F}_{R_n}$ be the maximum number of minimal forts over all forests on $n$ vertices\footnote{We use ``$R$'' for forests to distinguish it from our notation for forts: ``$F$'' or ``$\mathcal{F}$''; ~\foreignlanguage{vietnamese}{``Rừng''} is Vietnamese for ``Forest''.}.

\subsection{Facts and Results about Minimal Forts on Forests and Trees}



\begin{proposition}[See \cite{becker2025number}] \label{prop:path}
For $n\ge 4$,
\[\mathcal{F}_{P_n} = \mathcal{F}_{P_{n-2}} +  \mathcal{F}_{P_{n-3}}\]
That is, $\mathcal{F}_{P_{n}}$ obeys the recurrence relation $a_n = a_{n-2} + a_{n-3}$ with $a_1 = a_2 = a_3 = 1$.\hfill $\qed$
\end{proposition} 

The solution to this recurrence relation is $\mathcal{F}_{P_{n}} = k_1 \psi^n + k_2 \omega_2^n + k_3 \omega_3^n$ where we define $\psi \approx 1.32472$ to be the real root to $z^3-z-1 = 0$ and the other bases \[ \omega_2, \omega_3 = -\frac{\psi}{2} \pm \frac{i}{2}\sqrt{\frac{3-\psi}{\psi}} \approx -0.662359 \pm 0.56228 i\] are the complex roots of the polynomial $z^3-z-1 = 0$. The coefficient $k_1 = \psi^4 / (2\psi+3) \approx 0.545116$ and the other coefficients
\[ k_2, k_3 = \frac{2 - 7\psi - 3\psi^2}{46} \mp i \left( \frac{7 - 3\psi}{46} \right) \sqrt{\frac{3-\psi}{\psi}} \approx -0.272558 \mp 0.0739727 i.\]
We will use $ \varepsilon_n := k_2 \omega_2^n + k_3 \omega_3^n$. An important property we use is $|\omega_2|=|\omega_3| < 1$ and $|k_2|=|k_3|<1$.  A basic calculation has $|\varepsilon_n| < 1$ for all $n > 0$; this yields the following:

    \begin{proposition}
     Let $k_1 = \frac{\psi^4}{2\psi+3}$. For any $n \ge 0$,
        \[ k_1 \psi^n - 1 < \mathcal{F}_{P_{n}} < k_1 \psi^n + 1.\] 
    \end{proposition}
    \hfill \qed

\begin{proposition}[See \cite{becker2025number}]
For $n \ge 3$, $\mathcal{F}_{S_n} = {{n-1}\choose 2}$
\end{proposition}

Recall that $T(n,k,m,p)$ is the tree of height 2 with root $r$ where root $r$ has $k$ children, among those $k$ children $k-p$ have $m$ children and $p$ have $m-1$ children. As it turns out for specific parameters,  $T(n,k,m,p)$ has more minimal forts than the path graph.

\begin{proposition}
Choose integers $k \ge 2$, $m \ge 3$, and $0 \le p \le k$. Let $n = 1 + k + (k-p)m + p(m-1)$. Then,

\[ \mathcal{F}_{T(n,k,m, p)} = m^{k-p} (m-1)^p + (k-p) {m \choose 2} + p {{m-1} \choose 2}. \]
\end{proposition}

\begin{proof}
Let $r$ denote the root; let $b_{1}, \ldots, b_{k}$ denote its $k$ children; let $a_{i,1}, a_{i,2}, a_{i,3}, \ldots,  a_{i,m}$ denote the children of $b_i$ for $i \le k-p$; let $a_{i,1}, a_{i,2}, a_{i,3}, \ldots,  a_{i,m-1}$ denote the children of $b_i$ for $i > k-p$; 


Let $F$ be a minimal fort. If $b_i \in F$, then by Corollary \ref{lem:three_uncolored_in_a_row}, at most one of its neighbors is in $F$. This would mean that at least one of $a_{i,1}, a_{i,2}, a_{i,3}, \ldots,  a_{i,\ell}$ is not in $F$. However, this violates the definition of a fort as that leaf neighbor has exactly one neighbor in $F$, $b_i$. We can assume that $b_i \not\in F$ for every $i$.

There are two cases: either $r \in F$ or $r \notin F$. 

If $r \in F$, then because of Corollary \ref{cor:treenot3neighbors}, every $b_i$ must have exactly one leaf neighbor $a_{i,\ell} \in F$. Such a selection is a fort as each $b_i$ is adjacent to exactly two vertices in $F$ and each $a_{i,m}$ not chosen is only adjacent to $b_i$ which is not in $F$. This fort is minimal as coloring $r$ will allow each $b_i$ to force their leaf neighbor in $F$; likewise, coloring any $a_{i,\ell} \in F$ will allow $b_i$ to force $r$ which will allow each remaining $b_j$ to force their leaf neighbor in $F$.

If $r \not\in F$, then since $F$ is nonempty, there must be at least one leaf $a_{i,j} \in F$. By Lemma \ref{lem:cutvertex}, since $r, b_i \not\in F$, there must be exactly two $a_{i,j}, a_{i,j'} \in F$ as the only elements in $F$. This selection is a fort as only $b_i$ is adjacent to vertices in $F$ and is adjacent to exactly two of them. This set is minimal as coloring either vertex will allow $b_i$ to force the other.

Notice there are $m^{k-p} (m-1)^p$ ways to make the first selection and there are $(k-p) {m \choose 2} + p {{m-1} \choose 2}$ ways to make the second selection and these selections are necessarily disjoint.

\end{proof}

For $n$ with $n-1$ divisible by $5$, we have $\mathcal{F}_{T(n,\frac{n-1}{5},4,0)} = 4^{(n-1)/5} + \frac{n-1}{5} \cdot 6 \approx C \cdot 1.31951^n$ with 
$C = 4^{-1/5} \approx 0.757858$. To compare with $P_n$, $ \mathcal{F}_{P_n} \approx k_1 \cdot 1.32472^n$ 
with $k_1 = \frac{\psi^4}{2\psi+3} \approx 0.545116$. Hence, while $\mathcal{F}_{T(n,\frac{n-1}{5},4,0)}$ grows slower than $\mathcal{F}_{P_n}$ for $n$ large, the coefficient $4^{-1/5} \approx 0.757858$ allows $\mathcal{F}_{T(n,\frac{n-1}{5},4,0)}$ to be larger than $\mathcal{F}_{P_n}$ for smaller $n$. 

Indeed, for $p=0, 1, 2, 3, 4$, one can numerically verify that for valid choices of $n$, $\mathcal{F}_{T(n,\frac{n-1+p}{5},4,p)} \ge \mathcal{F}_{P_n}$ for all $n$ up to $n=73$. Since all $n$ are possible with at least one of these choices of $p$, the path graph is not the tree with the maximum number of minimal forts until at least $n=73$. 

Figure \ref{fig:startreepathplot} shows that $\mathcal{F}_{T(n,\frac{n-1}{5},4,0)}$ overcomes $\mathcal{F}_{S_n}$ near $n \approx 20$. In Section \ref{sec:algorithm}, based on the results of an algorithm, we will show that indeed the unique tree achieving $\mathcal{F}_{T_{19}}  = 162$ is $T(19,4,4,2)$. 

\begin{figure}
\centering
\includegraphics[width=0.5\textwidth]{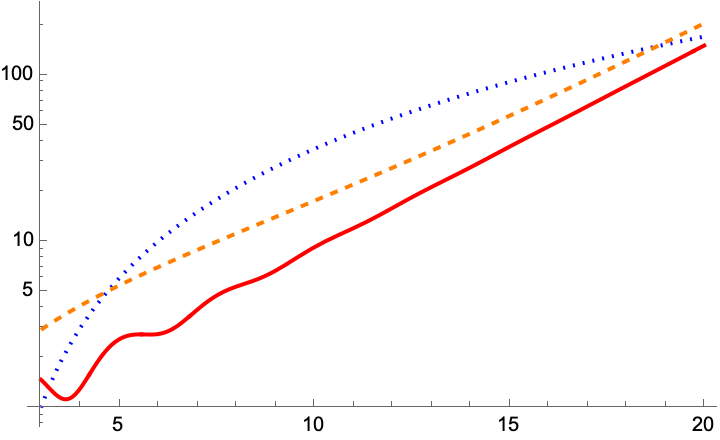}
\caption{Continuous log-plot of the functions describing the number of minimal forts of different families of trees including stars, $S_{n}$ (dotted blue), the trees $T(n,\frac{n-1}{5},4,0)$ (dashed orange) and paths, $P_n$ (solid red).}
\label{fig:startreepathplot}
\end{figure}

\subsection{Properties of Minimal Forts}
\label{sec:lemmas}

While we are only concerned about trees, some of the following lemmas apply for any connected graph $G$.

\begin{lemma} \label{lem:cutvertex}
Let $G$ be a connected graph and let $v$ be a cut-vertex of $G$ where $G_1, G_2, \ldots, G_c$ are the connected components of $G - \{v\}$. Then for any minimal fort, $F$, of $G$ where $v \notin F$, at most two $G_i$ have $V(G_i) \cap F \ne \emptyset $.
\end{lemma}

\begin{proof}
If $v$ has no neighbors in $F$, then we will show that $F \cap (V(G_1) \cup V(G_2))$ is still a fort. Since $v$ is a cut vertex, any vertex $u \in V(G_i)$ for $i=1,2$ has $N_G(u) \setminus \{v\} = N_{G_i}(u) = N_{G_1 \cup G_2}(u)$. Therefore $(N_G(u) \setminus \{v\}) \cap F = N_{G_1 \cup G_2}(u) \cap F$. However, since $v \notin F$, $N_G(u) \cap F = N_{G_1 \cup G_2}(u) \cap F$. Hence, any vertex $u \in (V(G_1) \cup V(G_2)) \cap F$ will have the same number of neighbors in $F$ as $F \cap (V(G_1) \cup V(G_2))$.

If $v$ has at least two neighbors in $F$, without loss of generality, let those neighbors be in $V(G_1)$ and $V(G_2)$. 
In which case, the argument above still applies for any vertex $u \in (V(G_1) \cup V(G_2)) \cap F$, as well as for $v$.

In either case, it follows that $F \cap (V(G_1) \cup V(G_2))$ is a fort of $G$ which violates the minimality of $F$.
\end{proof}

\begin{cor} \label{cor:treenot3neighbors}
Let $T$ be a tree, for any minimal fort $F$, there does not exist distinct $a, b, c \in F$ and $u \in V(T)$ with $a, b, c \in (N(u) \cup \{u\})$.
(i.e., for any $u \in V(T)$, $|(N(u) \cup \{u\}) \cap F| = 0$ or $2$.)
\end{cor}

\begin{proof}
If $u$ is a leaf, then if $u \not\in F$, its neighbor cannot be in $F$, so the result holds.

If $u$ is not a leaf and $u \notin F$, then since every non-leaf vertex of $T$ is a cut vertex, we can apply Lemma \ref{lem:cutvertex} and at most two of its neighbors can be in $F$. 

If $u$ is not a leaf and $u \in F$, then $u$ is adjacent to other distinct vertices $a$ and $b$. Observe that coloring $u$ may allow forces to occur, but these forces cannot force $a$ nor $b$ as $T$ is a tree and $u$ still has two uncolored neighbors. The resulting failed forcing set yields a smaller fort, violating the minimality of $F$.
\end{proof}

Corollary \ref{cor:treenot3neighbors} also implies following related statement.

\begin{cor} \label{lem:three_uncolored_in_a_row}
Let $T$ be a tree. If $F$ is a minimal fort of $T$, then $F$ does not contain three vertices $a,b,c$ with $a \sim b$ and $b \sim c$. \hfill $\qed$
\end{cor}

\begin{lemma} \label{lem:cutedge}
Let $G$ be a connected graph and let $e = \{a,b\}$ be a cut-edge of $G$. Let $G_a$ and $G_b$ denote the connected components of $G$ with the edge $e$ removed. Then for any minimal fort, $F$, of $G$ where $a, b \notin F$, either $V(G_a) \cap F = \emptyset$ or $V(G_b) \cap F = \emptyset$.
\end{lemma}

\begin{proof}
If $V(G_a) \cap F \ne \emptyset$ and $V(G_b) \cap F \ne \emptyset$, then since $\{a,b\}$ is a cut edge, no vertex in $V(G_a)$ is adjacent to a vertex in $V(G_b)$; hence, $F \setminus V(G_a)$ (i.e., coloring $G_a$) is still a fort of $G$ which violates the minimality of $F$.
\end{proof}

\begin{lemma} \label{lem:forestdecomposition}
    Let $G$ be a graph with connected components, $G_1, \ldots, G_c$. Then, $\mathcal{F}_G = \mathcal{F}_{G_1} + \cdots + \mathcal{F}_{G_c}$.
\end{lemma}

\begin{proof}
    Let $F$ be a fort on $G$. If $F \cap V(G_i)$ and $F \cap V(G_j)$ are nonempty for $i\ne j$, then consider $F \cap V(G_i)$. Any vertex not in $F$ but in $V(G_i)$ is still adjacent to 0 or 2 or more vertices in $F \cap V(G_i)$ and any vertex not in $F$ nor in $V(G_i)$ is necessarily adjacent to no vertices in $F \cap V(G_i)$. Hence $F \cap V(G_i)$ is a fort of $G$. It follows that any minimal fort of $G$ is a minimal fort of some $G_i$ and the result follows.
\end{proof}

In the context of trees, this means for a forest $R$ with connected components, $T_1, \ldots, T_c$ we have $\mathcal{F}_R = \mathcal{F}_{T_1} + \cdots + \mathcal{F}_{T_c}$.

The next lemma characterizes exactly which sets of a tree are minimal forts solely based on their graph-theoretical properties. As it turns out, a very similar characterization was also independently found by Cameron and Li (See Theorem 3.3 in \cite{cameron2025minimal}). We include this lemma for completeness. Our  characterization is slightly different and arguably more complicated; however, this presentation lends itself more naturally to the algorithms we present later as it carefully delineates all of the possibilities for extending potential minimum fort based on the inclusion and/or exclusion of two adjacent vertices.

\begin{lemma} \label{lem:lemmaforalgorithm}
Let $T$ be a tree on at least three vertices.
Choose $S \subseteq V(T)$. 
Then, $S$ is a minimal fort if and only if the following conditions hold:

\begin{enumerate} [label=(\Roman*)]

\item $S$ contains at least one leaf; and 
\label{lemalg_condA}

\item For any leaf $\ell \in S$, every edge $\{a, b\}$ where $a$ is closer to $\ell$ than $b$ (it may be that $a = \ell$ or $b$ is a leaf) satisfies the following: \label{lemalg_condB}

\begin{enumerate}[label=(\roman*)]
    \item If $a, b \notin S$, $N(b) \cap S = \emptyset$. \label{lemalg_cond1}
    \item If $a \notin S$ but $b \in S$, $|N(b) \cap S| \le 1$. \label{lemalg_cond2}
    \item If $a \in S$ but $b \notin S$, $|(N(b) \cap S) \setminus\{a\}| = 1$. \label{lemalg_cond3}
    \item If $a, b \in S$, $|(N(b) \cap S) \setminus\{a\}| = 0$. \label{lemalg_cond4}
\end{enumerate}
\end{enumerate}


\end{lemma}

Before we prove Lemma \ref{lem:lemmaforalgorithm}, we need the following lemma and definition.

\begin{lemma} \label{lem:twoleafpath}
Any set $S \subseteq V(T)$ that meets conditions \ref{lemalg_condA} and \ref{lemalg_condB} has the property that any vertex $v \in S$ lies on a path between two leaves $\ell, \ell' \in S$.
\end{lemma}

\begin{proof}

 By condition \ref{lemalg_condA}, there is at least one leaf $\ell \in S$. Suppose for a contradiction there is a vertex $p \in S$ such that it is not on a path between two leaves in $S$. Let $p$ be the furthest such vertex away from $\ell$. Note that necessarily $p$ cannot be a leaf. So $p$ must have another neighbor $p'$ that is further from $\ell$ than $p$. However, $p' \not \in S$, so applying condition \ref{lemalg_cond3} requires that $p'$ has a neighbor $p''$ in $S$. Necessarily, $p''$ is further from $\ell$ than $p$, a contradiction.
\end{proof}

\begin{definition}
Let $T$ be a tree and let $S \subseteq V(T)$ be a subset of $V(T)$. The \emph{forcing graph} of $S$ is the directed graph $H=(S,E)$ where $s_1 \to s_2$ if upon coloring $s_1$ in addition to $\overline S$, $s_2$ is forced (not necessarily by $s_1$) in the first simultaneous application of the color-change rule.
\end{definition}

\begin{proof}[Proof of Lemma \ref{lem:lemmaforalgorithm}]
($\Leftarrow$) Let $S$ be a set of vertices that satisfies the hypotheses. 

We will first show that $S$ is a fort. Suppose $S$ is not a fort, so there is a vertex $b \notin S$ with $b$ adjacent to exactly one neighbor in $S$, $a$. 
By Lemma \ref{lem:twoleafpath}, there are two leaves $\ell, \ell' \in S$ whose path between them contains $a$. Necessarily, choose $\ell$ to be closer to $a$ than $b$. Hence,
condition \ref{lemalg_cond3} applied to $\{a,b\}$ mandates that $b$ has another neighbor in $S$, a contradiction. 

To show that $S$ is minimal, we will show that coloring $\overline S \cup \{s\}$ for any $s \in S$ will color the entire tree. 

Let $H$ be the forcing graph of $S$. It suffices to prove $H$ is strongly connected, as then coloring any one additional vertex in $S$ (in addition to $\overline S$) will force the entire graph by the definition of $H$. 

Suppose $x, y \in S$ with $x \sim y$ and $x$ closer to $\ell$ than $y$. Then, condition \ref{lemalg_cond4} requires that $N(y) \cap S \setminus \{x\}  = \emptyset$. Therefore, if $y$ is colored in addition to $\overline S$, $y$ will force $x$. 
Hence, $y \to x$. 

If $x$ is not a leaf in $T$, then there is a vertex closer to $\ell$ from $x$, $w$. It must be the case that $w \notin S$, otherwise condition \ref{lemalg_cond4} is violated when $a= w$, $b = x$ as $y \in N(x) \cap S \setminus \{w\}$. But since $w \notin S$, we can apply condition \ref{lemalg_cond2} and conclude that $y$ is the only element of $N(x) \cap S$. Therefore if $\overline S \cup \{x\}$ is colored, we have that $x$ will force $y$, so $x \to y$ in $H$. If $x$ is a leaf in $T$, then coloring $x$ will necessarily force $y$, so $x \to y$.

If $x, z \in S$ that are distance 2 in $T$, then condition \ref{lemalg_cond4} requires that their mutual neighbor, $y$ has $y \notin S$. Since $x \in S$ and $y \not \in S$, condition \ref{lemalg_cond3} guarantees that there is at most one other neighbor $y$ that is in $S$ which must be $z$. Hence, coloring $x$ will allow $y$ to force $z$ and we have $x \to z$ in $H$. By same argument in reverse, we also have that $z \to x$ in $H$. Hence any vertices of distance 2 in $T$ are mutually adjacent in $H$.

If $H$ is disconnected, then there are two vertices $x', y'$ in different connected components of $H$ that have minimum distance in $T$. Necessarily, that minimum distance must be 3 (otherwise they would be connected in $H$). This means there are vertices $a, b \notin S$ that are on the path between $x'$ and $y'$. However, this violates condition \ref{lemalg_cond1}.

It follows that $H$ is connected and therefore, coloring any one additional vertex in $S$ in addition to all of $\overline S$ will force all of $T$. Hence, $S$ is a minimal fort.

($\Rightarrow$) Suppose $S$ is a minimal fort. Necessarily, $S$ must contain a leaf (see \cite{cameron2023forts}, Lemma 4.13), so it obeys condition \ref{lemalg_condA}.

For the remaining conditions, we note the following.
Corollary \ref{lem:three_uncolored_in_a_row} requires that a minimal fort satisfies conditions \ref{lemalg_cond2} and \ref{lemalg_cond4}. Corollary \ref{cor:treenot3neighbors} requires that a minimal fort satisfy condition \ref{lemalg_cond3} (noting that $b$ already has a neighbor in $S$, so it must have another). Lemma \ref{lem:cutedge} requires that a minimal fort satisfy condition \ref{lemalg_cond1} (note that there must be $\ell\in V(G_t), t\in\{a,b\}$ such that $V(G_t) \ne \emptyset$).
\end{proof}

\begin{lemma} \label{lem:bothways}
Suppose a set $S$ satisfies condition \ref{lemalg_condB} for a leaf $\ell = \ell' \in S$, then $S$ also satisfies condition \ref{lemalg_condB} for any other leaf with $\ell = \ell'' \in S$.
\end{lemma}

\begin{proof}
Let $S$ be a set that satisfies condition \ref{lemalg_condB} when choosing $\ell = \ell' \in S$. Let $\ell'' \in S$ be another leaf in $S$. Consider an edge $\{a,b\}$ in $T$ with $a$ closer to $\ell'$ than $b$. Observe that $a$ will still be closer to $\ell''$ unless $\{a,b\}$ in on the shortest path between $\ell'$ and $\ell''$ , $P$. Hence condition \ref{lemalg_condB} is satisfied for $\ell = \ell''$ for any edge not on $P$. It remains to show the condition holds for any edge $\{a', b'\}$ on the shortest path $P$ with $a'$ closest to $\ell''$.

If $P$ only has one edge then $T$ is the graph on a single edge and condition $\ref{lemalg_condB}$ holds.

Next, consider the edge $\{a' , b'\}$ when $b' = \ell'$ and $a'$ is the sole neighbor of $b'$ (and hence $a'$ is closer to $\ell''$ than $b'$). In this  case, if $a' \not\in S$, condition \ref{lemalg_cond2} for $\ell = \ell''$ with the edge $\{a',b'\}$ applies as $b'$ has at most one neighbor $a'$. If $a' \in S$, condition \ref{lemalg_cond4} for $\ell = \ell''$ with the edge $\{a',b'\}$ applies as $b'$ has no neighbors other than $a'$.

We can now assume that $P$ has at least three vertices and that $b' \ne \ell'$.
Choose any three adjacent vertices $a \sim b (=b') \sim a'$ on $P$ with $a$ closest to $\ell'$ (it may be the case that $a' = \ell''$). We will show that for any combination of whether $a , b (=b')$ and $a'$ are in $S$ or not in $S$, $\ell'$ satisfying condition \ref{lemalg_condB} implies that $\ell''$ satisfies the conditions with respect to the edge $\{a', b'\}$.

{\bf Case 1}: $a, b \not\in S$ or $b', a' \not\in S$:

If $a, b \not\in S$ or $b', a' \not\in S$, then applying condition \ref{lemalg_cond1} with $\ell = \ell'$ inductively, we have $\ell'' \not\in S$, a contradiction.

{\bf Case 2}: $a, a' \not\in S, b \in S$:

Applying condition \ref{lemalg_cond2}, there is at most one additional neighbor of $b$ (that is not $a'$) that is in $S$. Hence, with $a' \not\in S$ and $b' \in S$, the edge $\{a', b'\}$ satisfies condition \ref{lemalg_cond2} for $\ell=\ell''$.

{\bf Case 3}: $a \not\in S, b, a' \in S$:

Applying condition \ref{lemalg_cond2}, there are  no additional neighbors of $b$ that are in $S$. Hence, with $a', b' \in S$, the edge $\{a', b'\}$ satisfies condition \ref{lemalg_cond4} for $\ell=\ell''$.

{\bf Case 4}: $a, a' \in S, b \not\in S$:

Applying condition \ref{lemalg_cond3}, there are no  additional neighbors of $b$ that are in $S$. Hence, with $a' \in S$ and $b' \not\in S$, the edge $\{a', b'\}$ satisfies condition \ref{lemalg_cond3} for $\ell=\ell''$.

{\bf Case 5}: $a, b \in S, a' \not\in S$:

Applying condition \ref{lemalg_cond4}, there are no  additional neighbors of $b$ that are in $S$. Hence, with $a' \notin S$ and $b' \in S$, the edge $\{a', b'\}$ satisfies condition \ref{lemalg_cond2} for $\ell=\ell''$.

{\bf Case 6}: $a, b, a' \in S$: 

Applying condition \ref{lemalg_cond4} to the edge \{a, b\} for $\ell = \ell'$ requires that $a' \not\in S$, so this case cannot occur.

\end{proof}

\begin{cor} \label{cor:condition_2_with_ending_leaf}
All sets that satisfy condition \ref{lemalg_condB} must have all the leaves whose only neighbor is uncolored, be uncolored.
\end{cor}

\begin{proof}
We assume that there is a colored leaf $\ell$ whose only neighbor $v$ is uncolored. Then Lemma \ref{lem:twoleafpath} says that there must exist another leaf $\ell'$ such that $v$ lies on the path between $\ell$ and $\ell'$. With the leaf $\ell'$ and $a=v,b=\ell$, the condition \ref{lemalg_cond3} says that there must be $|(N(b) \cap S) \setminus\{a\}| = 1$. But because $b=\ell$ is a leaf, we have $|(N(b) \cap S) \setminus\{a\}| = 0$, a contradiction.
\end{proof}

\begin{lemma} \label{lem:three_vertex_branch}
Let $G$ be a graph. Let $v, w, x$ be vertices of $G$ such that $v$ is only adjacent to $w$ and $w$ is only adjacent to $v$ and $x$. If $F$ is a minimal fort of $G$ such that $v, w \in F$, then $F \setminus \{v, w, x\}$ is a minimal fort of $G - \{v, w, x\}$.
\end{lemma}

\begin{proof}
Let $F$ be a maximal fort of $G$ with $v, w \in F$. Notice that it must be the case that $x \not \in F$ as otherwise $F \setminus \{w\}$ is a smaller fort. Hence, after deleting $\{v, w, x\}$, any vertex $u \not\in F \setminus \{v, w, x\}$ still has the same number of neighbors in $F$. Therefore, $F \setminus \{v, w, x\}$ is a fort of $G - \{v, w, x\}$. It remains to show that $F \setminus \{v, w, x\}$ is a minimal fort of $G - \{v, w, x\}$.

Suppose there is a smaller fort of $G - \{v, w, x\}$, $F' \subseteq F \setminus \{v, w, x\}$. Then either $x$ is adjacent to a vertex in $F'$ or not. If so, $F' \cup \{v, w\}$ is a fort of $G$ (as then $x$ has two neighbors in $F' \cup \{v, w\}$). If not, $F'$ is a fort of $G$. Either way, this violates the minimality of $F$.
\end{proof}

\begin{lemma} \label{lemma:path_branch_have_no_neighbor}
Let $T$ be a tree. If $T$ contains a path branch that has no neighbor then $T$ is a path. 
\end{lemma}

\begin{proof}
By definition a path branch is a maximal connected induced subgraph of $T$ consisting of only vertices of degree $1$ or $2$ in $T$. Since $T$ is connected, any path branch that is not the entire graph itself must contain a neighbor. Hence, the only possible way to be a neighbor is if the path branch is the entire graph itself. In which case, it must be a path graph.
\end{proof}

\begin{lemma} \label{lem:elim_a_path_branch}
Let $T$ be a tree that is not a path, $\ell$ be a leaf of tree $T$, $P_\ell$ be the path branch in tree $T$ that contains the leaf $\ell$, $v$ be the neighbor of path branch $P_\ell$ and $X$ be an arbitrary set of vertices satisfying $\ell\notin X$. Let $\mathcal{F}_{H,-A}$ denote the number of minimal forts of the graph $H$ which avoid the set $A$. Then, the number of minimal forts of $T$ which avoid $X$ can be determined recursively by
\[
\mathcal{F}_{T,-X}=\mathcal{F}_{T,\ell,-X} + \mathcal{F}_{T-P_\ell,-(X+v)} 
\]

\end{lemma}

\begin{proof}
Let $S$ be a minimal fort of $T$ with $S \cap X = \emptyset$. Choose a leaf $\ell$. Then either $\ell \in S$ or $\ell \not\in S$. In the first case, $S$ is counted toward $\mathcal{F}_{T,\ell,-X}$. If $\ell \not\in S$ then the color-change rule forces the entire path branch of $\ell$ and its neighbor, $v$ to be colored. Since no vertex in $P_\ell$ is adjacent to any vertex in $T-P_\ell$ except $v$, $S$ must also be a fort of $T - P_\ell$ that avoids $X+v$. This is a bijection since taking any fort of $T - P_\ell$ where $X+v$ is avoided is also a fort of $T$. \end{proof}

Remark: As presented, Lemma \ref{lem:elim_a_path_branch} counts the number of sets. However, the proof also applies in recursively determining all of the minimal forts of $T$, not just counting them.

\section{The Base Case for Theorem \ref{thm:main} ($n \le 19$) and an Algorithm for finding Minimal Forts of Trees}

In this section, we present the exact values for $\mathcal{F}_{T_n}$ and $\mathcal{F}_{R_n}$ up to $n=19$ as needed in the base case for the proof of Theorem \ref{thm:main}. These values are computed using an algorithmic search in Section \ref{sec:basecase}.

We emphasize that Theorem \ref{thm:main} does not hold for $n=2$ as $\mathcal{F}_{P_2} = 1$ whereas $\mathcal{F}_{R_2} = 2$. This will create a special subcase for our induction step in Section \ref{sec:mainresult}.

\subsection{An Efficient Algorithm for Listing all Minimal Forts of a Tree}\label{sec:algorithm}

In order to calculate these values, we present an efficient method for determining all minimal forts of a tree. Algorithms \ref{alg:minimal_forts} and \ref{alg:decide} operate together to generate a comprehensive list of minimal forts by systematically exploring the tree via breadth-first search. Algorithm \ref{alg:minimal_forts} controls the search and Algorithm \ref{alg:decide} implements Lemma \ref{lem:lemmaforalgorithm} to determine which unexplored neighbors can be added to the fort. The procedure is summarized as follows:

\begin{enumerate}[label=\textbf{Step \arabic*:},leftmargin=1.5cm]

\item Begin with an empty set of permanently colored vertices and choose a starting leaf $\ell$ to be included in a fort.

\item  Iteratively grow a list, $Q_{forts}$, of all possible forts that include that leaf $\ell$. Do this by traversing the tree by breadth-first search. At each step, there is an active vertex $u$ and its previously visited neighbor $p$.
Lemma \ref{lem:lemmaforalgorithm} explicitly determines all of the possibilities of extending each $F \in Q_{forts}$:
    \begin{itemize}
        \item If both $u$ and $p$ are colored, all remaining neighbors of $u$ must be colored.
        \item If only $u$ is colored, exactly one remaining neighbor of $u$ must be colored.
        \item If only $p$ is colored, at most one remaining neighbor of $u$ can be colored.
        \item If both $u$ and $p$ are uncolored, all remaining neighbors of $u$ must be colored.
        \item If $u$ is uncolored, all its leaf neighbors must also be uncolored.
    \end{itemize}
    If no valid extensions of $F$ exist (because it violates the last condition or it requires a vertex we have designated as permanently colored to be uncolored), discard $F$ instead.  Save all valid fort extensions into $Q_{forts}$ and repeat until the entire graph is traversed.

    \item Save all resulting forts $Q_{forts}$. Prune the tree by deleting the path branch containing $\ell$ and permanently color the neighbor of that path branch (preventing its inclusion in future forts) and repeat from step 1 to grow a new set of forts $Q_{forts}$ for a new leaf. 
    
    \item When entire tree is pruned (and hence all leaves are considered), return the entire collection of all forts.
\end{enumerate}

\label{alg:ComputeMinimalForts}
\begin{algorithm}[H]
\SetKwInOut{Input}{Input}
\SetKwInOut{Output}{Output}
\SetKwFunction{Decide}{Decide Neighbors Logic}
\SetKwFunction{Minimal}{Compute Minimal Forts}

\Input{A tree (described by  $Adj$, a list of lists of adjacency relations),\\
A set of prohibited vertices $colorForcedSet$ (default= $\emptyset$)}
\Output{List of Minimal Fort sets that do not contain vertices in $colorForcedSet$  
}

\caption{Compute Minimal Forts}
\label{alg:minimal_forts}

$Leaves \leftarrow $ all leaves of the tree\;
$Q_{v} \leftarrow$ empty queue, $Q_{prev} \leftarrow$ empty queue, $Q_{forts} \leftarrow$ empty queue\;

Pick a starting leaf $\ell \not \in colorForcedSet$, add it to $Q_{v}$, add $\{l\}$ into $Q_{forts}$ (i.e., $\ell$ is necessarily uncolored for this iteration) (
if there are no leaves left, return $\emptyset$ \label{begin}) \;

\tcp{We will now process vertices via breadth-first search}
\While{$Q_{v}$ is not empty}{

    $curr \leftarrow$ The next vertex from the queue $Q_v$; remove it from $Q_v$\;  
    $prev \leftarrow$ the next vertex from the queue $Q_{prev}$; remove it from $Q_{prev}$; (if none, $prev \leftarrow {null}$) \; 

    \tcp{$prev$ is a neighbor of $curr$ that was already visited.}

    Mark $curr$ as visited\;
    

    $Neighbors_{leaves} \leftarrow $ all leaves neighbors of $curr$\;

    $\mathcal{F}\leftarrow Q_{forts}$
\ForEach{$CurrentFort\in\mathcal{F}$}{
        remove $CurrentFort$ from $Q_{forts}$     
    \label{lin10} \; 
        \tcp{Branch out solutions based on local constraints by applying $\Decide$ (Algorithm \ref{alg:decide})} 
        Add all sets $\Decide(F=CurrentFort, N=Adj[curr],colorForcedSet,L= Neighbors_{leaves}, p=prev, u=curr)$ to $Q_{forts}$\;
        }
    Mark all unvisited neighbors and put them into $Q_{v}$; for each one, add its visited neighbor, $curr$, into $Q_{prev}$\;
}
\tcp{Handle remaining branches recursively}
\If{$Adj$ is empty}{
\Return{$\emptyset$} \quad\tcp{i.e., if the tree is empty; this is the base case of the recursion}
}

$newTreeAdj \leftarrow$ Adjacency list for $T - B_\ell$, where $B_\ell$ is the path branch containing starting leaf $\ell$\label{deletepath}\;

$ colorForcedSet \leftarrow colorForcedSet \cup \{ \text{the neighbor of the path branch } B_\ell\} \label{addtoforce}$;


\Return{ $Q_{forts}\cup$ \Minimal (newTreeAdj,colorForcedSet)} \;
\end{algorithm}

\begin{algorithm}[H]
\label{alg:decide}
\SetKwInOut{Input}{Input}
\SetKwInOut{Output}{Output}
\Input{Fort so far $F$, Neighbors $N$, Forced Set $colorForcedSet$, Leaf neighbors $\ell$, Previous visited neighbor $p$, Center $u$}
\tcp{$N$ is set of all neighbors of $u$}

\Output{List of Valid Fort Extensions}

$List \leftarrow \emptyset$\;
\tcp{If $u$ is a starting leaf from Algorithm I, it is not bound by any conditions from (II), so a possible fort can possibly contain $\{u\}$ (but not $N[0]$) xor possibly $\{u,N[0]\}$}
\If{$p$ is null}{
    $List\leftarrow List \cup \{\{u\}\}$\;
    \If{$N[0]\notin colorForcedSet$}{
$List\leftarrow List \cup \{\{u,N[0]\}\}$\;
    }
    \Return{$List$}\;
}
\tcp{if $u$ is not a starting leaf, apply the conditions from Lemma \ref{lem:lemmaforalgorithm}.}
\uIf{$u\notin F$ (i.e., $u$ is colored)}{
    \eIf{$p\notin F$ (i.e, $p$ is colored)}{
        \tcp{Apply condition (II.i): There can be no additional neighbors in $F$}
        \Return{$\{F\}$}\;
    }{
        \tcp{Apply condition (II.ii): There must one additional neighbor in $F$}
        \ForEach{$n \in N$}{
            \If{$n \neq p \land n \notin colorForcedSet$}{
                $List\leftarrow List \cup \{F \cup \{n\}\}$\;
            }
        }
    }
}
\Else{
    \eIf{$p\notin F$ (i.e, $p$ is colored)}{

        \uIf{$|L| > 1$}{
        \tcp{Apply condition (II.iii): At most one leaf must be uncolored, making the rest colored, which cannot happen; remove that potential fort from the list}
            \Return{$null$}\;
        }
        \uElseIf{$|L| = 1$}{
        \tcp{Apply condition (II.iii) and condition (I): The leaf must be uncolored!}
            \Return{$\{F \cup \{L[0]\}\}$}\;
        }
        \Else{ 
            \tcp{Apply condition (II.iii): There can be at most one additional neighbor in $F$}
            $List\leftarrow List \cup \{F\}$\;
            \ForEach{$n \in N$}{
                \If{$n \neq p \land n \notin colorForcedSet$}{
                    $List\leftarrow List \cup \{F \cup \{n\}\}$\;
                }
            }
        }
    }{
        \uIf{$(|L| = 1 \land L[0] \neq p) \lor |L| > 1$}{   \tcp{Apply condition (II.iv): There can be no additional neighbors in $F$, which cannot happen if any remaining unvisited neighbors are leaves}
            \Return{$null$}\;
        }
         \tcp{Otherwise apply condition (II.iv): No additional neighbors can be added.}
        \Return{$\{F\}$}\;
    }
}
\Return{$List$}\;
\caption{Decide Neighbors Logic (a subroutine of Algorithm \ref{alg:minimal_forts})}
\end{algorithm}

\begin{theorem}
\Minimal($T, \emptyset$) (Algorithm \ref{alg:minimal_forts}) returns exactly the set of all minimal forts of the tree $T$.
\end{theorem}

\begin{proof} 

Observe that \Decide emulates the necessary and sufficient conditions of Lemma \ref{lem:lemmaforalgorithm} (and Corollary \ref{cor:condition_2_with_ending_leaf}):

\begin{itemize}
        \item Condition \ref{lemalg_condB}\ref{lemalg_cond1}: If both $u$ and $p$ are colored, all remaining neighbors of $u$ must be colored.
        
        \item Condition \ref{lemalg_condB}\ref{lemalg_cond2}: If only $u$ is colored, there must be exactly one among all remaining neighbors of $u$ are colored.
        
        \item Condition \ref{lemalg_condB}\ref{lemalg_cond3}: If only $p$ is colored, there can only be at most one among all remaining neighbors of $u$ is colored.

        \item Condition \ref{lemalg_condB}\ref{lemalg_cond4}: If both $u$ and $p$ are uncolored, all remaining neighbors of $u$ must be colored.

    \end{itemize}

Note that for conditions \ref{lemalg_cond3} and \ref{lemalg_cond4}, if the result is not possible, \Decide will remove that set as a possible subset of a minimal fort (Corollary \ref{cor:condition_2_with_ending_leaf}).

Altogether by Lemma \ref{lem:lemmaforalgorithm}, the subroutine \Decide produces all possible subsets of minimal forts that contain $F$ and avoid ${coloredForcedSet}$. Hence, exhaustively and recursively calling \Decide by expanding each possible subset (\Minimal line \ref{lin10}) will delineate all minimal forts that avoid ${coloredForcedSet}$.

By Lemma \ref{lem:elim_a_path_branch}, it suffices to iterate this procedure by ordering the leaves of $T$, $\ell_1, \ldots, \ell_k$, and beginning each iteration with $F =  \{\ell_i\}$ (\Minimal line \ref{begin}), having deleted the path branches $B_1, \ldots, B_{i-1}$ (\Minimal line \ref{deletepath}) and adding the neighbors of each path branch $v_1, \ldots, v_{i-1}$ (which may be duplicated) to ${coloredForcedSet}$ (\Minimal line \ref{addtoforce}). (For emphasis, each path branch $B_i$ is a path branch of $T-B_1-B_2-\ldots-B_{i-1}$ not necessarily a path branch of $T$.)

\end{proof}

    \begin{figure} \centering
\begin{tikzpicture}[
    every node/.style={circle, draw, minimum size=4mm, inner sep=0pt, thick},%
    level 1/.style={sibling distance=32mm, level distance=15mm, thick}, %
    level 2/.style={sibling distance=8mm, level distance=15mm, thick},%
    edge from parent/.style={draw, thick}
]

\node {}
    child { node {}
        child { node {} }
        child { node {} }
        child { node {} }
        child { node {} }
    }
    child { node {}
        child { node {} }
        child { node {} }
        child { node {} }
        child { node {} }
    }
    child { node {}
        child { node {} }
        child { node {} }
        child { node {} }
    }
    child { node {}
        child { node {} }
        child { node {} }
        child { node {} }
    };

\end{tikzpicture}

\caption{The tree $T(19,4,4,2)$. It is the unique tree on 19 vertices with the maximum number of minimal forts (162) among all trees on $19$ vertices.}
    \end{figure}

\begin{figure} \centering

\begin{tikzpicture}[
        every node/.style={circle, draw, minimum size=4mm, inner sep=0pt, thick},
        %
        level 1/.style={sibling distance=20mm, level distance=15mm, thick},
        %
        level 2/.style={sibling distance=9mm, level distance=15mm, thick},
        %
        edge from parent/.style={draw, thick}
    ]

    \node [] {}
        child { node {}
            child { node {} }
            child { node {} }
            child { node {} }
        }
        child { node {}
            child { node {} }
            child { node {} }
        }
        child { node {}
            child { node {} }
            child { node {} }
        }
        child { node {}
            child { node {} }
            child { node {} }
        }
        child { node {}
            child { node {} }
            child { node {} }
        }
        child { node {} }
        child { node {} };

    \end{tikzpicture}

\caption{An example of a tree with the minimum number of minimal forts (8) among all trees on $19$ vertices.}
\label{fig:19minimum}
\end{figure}

\subsection{Explicit Results for the Base Case} \label{sec:basecase}

\begin{table}[ht]
    \centering
    \begin{tabular}{|c|c|c|c|c|}
        \hline
        $n$ & $\mathcal{F}_{T_n}$ & \text{Maximum Tree} & $\mathcal{F}_{R_n}$ & \text{Maximum Forest} \\
        \hline \hline
        4 & 3 & $S_4$ & 4 & $E_4$, Empty Graph\\
        5 -- 18 & $\binom{n-1}{2}$ & $S_n$ &  $\binom{n-1}{2}$  & $S_n$  \\
        19 & 162 & $T(19,4,4,2)$ & 162 & $T(19,4,4,2)$ \\
        20 & 213 & $T(20,4,4,1)$ & 213 & $T(20,4,4,1)$ \\
        
        \hline
    \end{tabular}
    \caption{A table of values for $\mathcal{F}_{T_n}$ and $\mathcal{F}_{R_n}$ for $n=4$ through $n=20$.}
    \label{tab:my_table}
\end{table}

\begin{lemma} \label{lem:basecase} For $n=1, \ldots 19$,
\[
\mathcal{F}_{T_n} = 
\begin{cases} 
    1 & \text{for } n = 1, 2, 3 \\
    \binom{n-1}{2} & \text{for } n = 4, 5, \ldots, 18 \\
    162 & \text{for } n = 19
\end{cases}
\]

\[
\mathcal{F}_{R_n} = 
\begin{cases} 
    n & \text{for } n = 1, 2, 3, 4\\
    \binom{n-1}{2} & \text{for } n = 5, \ldots, 18 \\
    162 & \text{for } n = 19
\end{cases}
\]
\end{lemma}

\begin{proof}
    The calculation for $\mathcal{F}_{T_n}$ is achieved by brute force search by computer by applying Algorithms \ref{alg:minimal_forts} and \ref{alg:decide} to the comprehensive list of trees compiled by McKay \cite{McKayTrees}; the cases $n=1, 2, 3$ are performed by inspection. The implementation and results of this computation can be found here \cite{datkony_minimalforts}.
    
    The calculation for $\mathcal{F}_{R_n}$ is achieved by applying the results for $\mathcal{F}_{T_n}$ to Lemma \ref{lem:forestdecomposition} by considering all integer partitions of $n$ with the {\it Mathematica} code 
    \begin{verbatim}
ftn[19]:= 162; ftn[n_] := If[n <= 3, 1, Binomial[n - 1, 2]]
Table[Max[(Total[(ftn[#] &) /@ #]) & /@ IntegerPartitions[i]], {i, 1,19}]\end{verbatim}
\end{proof}

\begin{lemma}
Theorem \ref{thm:main} holds for $3 \le n \le 19$.
\end{lemma}

The lemma follows from calculations based upon Lemma \ref{lem:basecase}  and Proposition \ref{prop:path}. The results are given in Table \ref{tab:n2path}. \hfill $\qed$

    \begin{table}[ht]
\centering

\begin{tabular}{|c|c|c|c|}
\hline
    $n$ & $\mathcal{F}_{T_n}$ & $\mathcal{F}_{R_n}$ & $\binom{n}{2} \mathcal{F}_{P_n}$ \\
\hline \hline
1  & 1   & 1   & 1      \\
\color{red}{\textbf{2}}  & \color{red}{\textbf{1}}   & \color{red}{\textbf{2}}   & \color{red}{\textbf{1}}      \\ \hline
3  & 1   & 3   & 3      \\
4  & 3   & 4   & 12     \\
5  & 6   & 6   & 20     \\
6  & 10  & 10  & 45     \\
7  & 15  & 15  & 84     \\
8  & 21  & 21  & 140    \\
9  & 28  & 28  & 252    \\
10 & 36  & 36  & 405    \\
\hline
\end{tabular} \quad
\begin{tabular}{|c|c|c|c|}
\hline
    $n$ & $\mathcal{F}_{T_n}$ & $\mathcal{F}_{R_n}$ & $\binom{n}{2} \mathcal{F}_{P_n}$ \\
    \hline \hline
11 & 45  & 45  & 660    \\
12 & 55  & 55  & 1,056  \\
13 & 66  & 66  & 1,638  \\
14 & 78  & 78  & 2,548  \\
15 & 91  & 91  & 3,885  \\
16 & 105 & 105 & 5,880  \\
17 & 120 & 120 & 8,840  \\
18 & 136 & 136 & 13,158 \\
19 & 162 & 162 & 19,494 \\ \hline
20 & 213 & 213 & 28,690 \\
\hline
\end{tabular}
 \caption{Results of calculations establishing the base case for Theorem \ref{thm:main} for $3 \le n \le 19$. Theorem \ref{thm:main} does not hold for $n=2$. The results for $n=20$ are not necessary for our base case.}
   \label{tab:n2path}
\end{table}

\subsection{Other Computational Results} \label{sec:compresults}

We apply the previous Algorithms to calculate the number of minimal forts on all trees from $n=4$ to $n=20$ vertices. 

Our implementation is completed in Java and can be found here: \cite{datkony_minimalforts}. Our computations were performed on a standard Windows laptop with an Intel i7-6820HQ processor with 4 cores, 2.7 GHz and 32 GB RAM.

Table \ref{tab:comp} shows the mean number of minimal forts and the mean execution time for trees on $n=10$ to $n=20$ vertices and Figure \ref{fig:execution} plots the individual execution times for all trees from $n=4$ to $n=20$ vertices. In general, the algorithm is very fast, as the {\it worst case} for all trees is just over 0.1 seconds, and in most cases, is substantially faster. The total computation lasts 984.4 seconds.

While we do not perform a formal complexity analysis, the worst cases for $n=16$ to $n=20$ follow a distinct increasing linear-trend, suggesting that the time complexity, in the worst case, increases exponentially.    

\begin{table}[ht!]
\centering
\caption{Mean number of minimal forts and execution time given by Algorithm \ref{alg:minimal_forts}.}
\begin{tabular}{|r|c|c|}
\hline
$n$ & Mean number of minimal forts & Mean execution time per tree (ms) \\ \hline \hline
{10} & 10.3019 & 0.0423525 \\ \hline
{11} & 11.9745 & 0.0346326 \\ \hline
{12} & 13.7731 & 0.0332337 \\ \hline
{13} & 15.8040 & 0.0393446 \\ \hline
{14} & 17.9984 & 0.0493654 \\ \hline
{15} & 20.4303 & 0.0544392 \\ \hline
{16} & 23.0982 & 0.0690336 \\ \hline
{17} & 26.0264 & 0.0934141 \\ \hline
{18} & 29.2348 & 0.1574270 \\ \hline
{19} & 32.7540 & 0.3378550 \\ \hline
{20} & 36.6056 & 1.0364120 \\ \hline
\end{tabular}
\label{tab:comp}
\end{table}

\begin{figure}[H]
    \centering
    \includegraphics[width=0.7\linewidth]{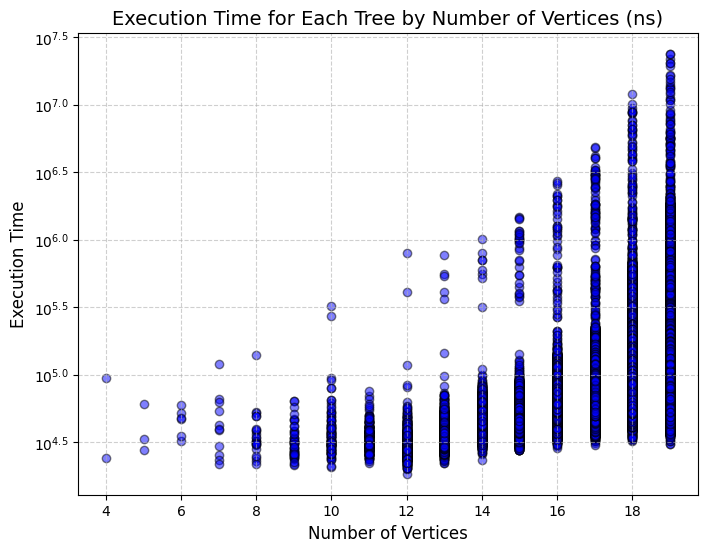}
    \caption{Log-plot of the execution time (ns) for individual trees based on the number of vertices.}
    \label{fig:execution}
\end{figure}

\section{Induction Step for Theorem \ref{thm:main} ($n \ge 20$) and a Core Recursion Inequality}

In this section, we provide the induction step for Theorem \ref{thm:main}, establishing the upper bound for the number of minimal forts on trees and forests.

\subsection{The Core Recursion Inequality for $\mathcal{F}_{R_n}$} \label{sec:recurr}

The main tool to prove Theorem \ref{thm:main} is a recursion inequality similar to the recursion relation for $\mathcal{F}_{P_n}$ that applies all forests.

\begin{lemma}\label{lem:recurr} For $n \ge 4$, $\mathcal{F}_{R_n}$ satisfies either
\begin{equation} \label{eqn:path}
 \mathcal{F}_{R_n} \le \mathcal{F}_{R_{n-2}}+\mathcal{F}_{R_{n-3}} 
\end{equation}
or, for some $d$, $3 \le d \le n-1$,
\begin{equation} \label{eqn:choices} \mathcal{F}_{R_n} \le {{d-1} \choose 2} + (d-1)\mathcal{F}_{R_{n-d}}.\end{equation}
\end{lemma}

\begin{proof}
Let $R$ be a forest on $n \ge 4$ vertices. If no connected component of $R$ has at least 4 vertices, then the number of minimal forts is equal to the number of connected components. In which case, $\mathcal{F}_R \le n$, and inequality (\ref{eqn:path}) holds as $\mathcal{F}_{R_k} \ge k$.

For the remainder, we will assume that $R$ has a component $T$ (a tree) with at least 4 vertices.

Choose a root, $r$ of $T$ so that the maximum distance from $r$ to any leaf vertex in $T$ is maximized. Since $n \ge 4$, $T$ necessarily has a leaf with a non-leaf neighbor. Let $v$ be such a leaf that is furthest from $r$. Let $w$ be the parent of $v$ and let $x$ be the parent of $w$. We have two cases: $d(w) = 2$ and $d(w) \ge 3.$

 {\bf Case 1: $d(w) = 2$.} There are three disjoint subcases.

{\bf Case 1.i:} $d(w)=2$ and $v, w \in F$. Then, $x \not \in F$ as otherwise $F \setminus \{w\}$ is a fort and so $F$ is not minimal. By Lemma \ref{lem:three_vertex_branch}, the set $F \setminus \{v,w,x\}$ is a fort of the subforest $R - \{v,w,x\}$ (which may be disconnected) on $n-3$ vertices. Hence, the number of forts of $T$  where $v, w \in F$ is at most $\mathcal{F}_{R_{n-3}}$.

{\bf Case 1.ii:} $d(w)=2$ and $w \not \in F$ but $v \in F$. Then, $x \not \in F$ as otherwise $\overline F$ is not a failed zero forcing set as $w$ will force $v$. It follows that $F \setminus \{v,w\}$ is a fort of the subforest $R - \{v,w\}$ with $n-2$ vertices. Hence, the number of forts where $w \not \in F$ but $v \in F$ is bounded above by the number of forts in a tree of $n-2$ vertices where a fixed vertex $x$ is included.

{\bf Case 1.iii}: $d(w)=2$ and $v, w \not \in F$. Then, $x \not \in F$ as otherwise $\overline F$ is not a failed zero forcing set as $w$ will force $x$. It follows that $F \setminus \{v,w\}$ is a fort of the subforest $R - \{v,w\}$ with $n-2$ vertices. Hence, the number of forts where $w \not \in F$ and $v \not \in F$ is bounded above by the number of forts in a tree of $n-2$ vertices where a fixed vertex $x$ is excluded.

The cases 1.ii and 1.iii together have that the total number of forts of $R$ where $w \not\in F$ is at most $\mathcal{F}_{R_{n-2}}$.
Altogether, since the three subcases of case 1 are disjoint, if $d(w)=2$, we have $\mathcal{F}_{R_n} \le \mathcal{F}_{R_{n-2}}+\mathcal{F}_{R_{n-3}}$.~\\

{\bf Case 2:} $d(w) \ge 3$. Let $d = d(w)$. If $w \in F$, then at most one neighbor of $w$ is in $F$; however, that means that one leaf neighbor of $w$ is not in $F$ (which exists as $d \ge 3$), and so $\overline F$ is not a maximal failed zero forcing set as that leaf can force $w$. Hence, $w \not \in F$. By Corollary \ref{cor:treenot3neighbors}, there are at most two neighbors of $w$ that are in $F$. We have three subcases:

{\bf Case 2.i:} There are two neighbors of $w$, $w_1$ and $w_2$ that are both leaves in $F$. By Corollary \ref{cor:treenot3neighbors} $x \not \in F$ and by Corollary \ref{lem:three_uncolored_in_a_row}, $w \not\in F$. Further, by Lemma \ref{lem:cutedge}, the only vertices in $F$ must be $w_1$ and $w_2$, and there are ${{d-1} \choose 2}$ ways to choose $w_1$ and $w_2$.

{\bf Case 2.ii:} There are two neighbors of $w$ in $F$, $w_1$ and $x$, where $w_1$ is a leaf in $T$. Observe that $F \setminus (N[v] \cup \{x\})$ is a fort of $T - (N[v] \cup \{x\})$.  Since there are $d-1$ ways of choosing $w_1$, the number of such forts is at most $(d-1)$ times the number of forts in $T - (N[v] \cup \{x\})$ where $x$ is included.

{\bf Case 2.iii:} There are no neighbors of $w$ in $F$; that is, $N[w] \cap F = \emptyset$. Observe that $F \setminus (N[v] \cup \{x\})$ is a fort of $T - (N[v] \cup \{x\})$. Hence, the number of such forts is at most the number of forts in $T - (N[v] \cup \{x\})$ where $x$ is excluded.

Cases 2.ii and 2.iii together have that the number of forts where at most one leaf neighbor of $w$ is included is at most $(d-1) \mathcal{F}_{R_{n-d}}$.

Altogether, since the three subcases are disjoint, if $d \ge 3$ we have \[ \mathcal{F}_{R_n} \le {{d-1} \choose 2}+ (d-1) \mathcal{F}_{R_{n-d}}.\] 
\end{proof}

\subsection{Combinatorial Inequalities for the Induction Step}

\begin{lemma} \label{lem:PndPnPsi}
For $n \ge 19$ and $d \in \{3, 4, 5\}$, the number of minimal forts on a path graph satisfies:
\[ \mathcal{F}_{P_{n-d}} \le \frac{101}{100} \frac{\mathcal{F}_{P_{n}}}{\psi^d} \]
\end{lemma}

\begin{proof}
Recall that the solution to the recurrence relation for the number of minimal forts of $P_n$ is given by $\mathcal{F}_{P_{n}} = \frac{\psi^{n+4}}{2\psi+3} + \varepsilon_n$, where $\varepsilon_n = k_2 \omega_2^n + k_3 \omega_3^n$ represents the contribution of the complex roots. 

We expand $\mathcal{F}_{P_{n-d}}$ as follows:
\begin{eqnarray*}
\mathcal{F}_{P_{n-d}} &=& \frac{\psi^{n+4-d}}{2\psi+3} + \varepsilon_{n-d} \\
&=& \frac{\psi^{n+4}}{2\psi+3} \cdot \frac{1}{\psi^d} + \varepsilon_{n-d} + \varepsilon_{n} \cdot \frac{1}{\psi^d} - \varepsilon_{n} \cdot \frac{1}{\psi^d} \\
&=& \left( \frac{\psi^{n+4}}{2\psi+3} \cdot \frac{1}{\psi^d}  + \varepsilon_{n} \cdot \frac{1}{\psi^d} \right) + \varepsilon_{n-d}- \varepsilon_{n} \cdot \frac{1}{\psi^d} \\
&\le& \frac{\mathcal{F}_{P_{n}}}{\psi^d} + |\varepsilon_{n-d}| +  |\varepsilon_n| \\
&\le& \frac{\mathcal{F}_{P_{n}}}{\psi^d} + | k_2 \omega_2^{n-d}| + |k_3 \omega_3^{n-d}|  +  |k_2 \omega_2^n| + |k_3 \omega_3^n| \\
&\le& \frac{\mathcal{F}_{P_{n}}}{\psi^d} + 4|k_2| |\omega_2|^{n-d} \\
\end{eqnarray*}

Where the third to fourth line follows from the definition of $\varepsilon_{n}$. The next line applies the fact that $\psi^d > 1$ with the triangle inequality. The second to last inequality applies the triangle inequality and uses the definition of $\varepsilon$. The last inequality applies the facts that $|\omega_2| < 1$, $k_2 = \overline{k_3}$ and $\omega_2 = \overline{\omega_3}$, so $|k_2| = |k_3|$ and $|\omega_2| = |\omega_3|$. 

For a fixed $d \ge 3$, since $|\omega_2| < 1$, the term $4|k_2| |\omega_2|^{n-d}$ is a strictly decreasing function of $n$, while $\mathcal{F}_{P_{n}}$ is an increasing function of $n$. To establish the lemma, it suffices to verify the values of $n$ for which $4|k_2| |\omega_2|^{n-d} \le \frac{\mathcal{F}_{P_n}}{100}$.

Computational verification yields the following results:
\begin{itemize}
    \item For $d=3$, the quantity $4|k_2| |\omega_2|^{n-d} - \frac{\mathcal{F}_{P_n}}{100}$ is negative for $n \ge 16$.
    \item For $d=4$, the quantity $4|k_2| |\omega_2|^{n-d} - \frac{\mathcal{F}_{P_n}}{100}$ is negative for $n \ge 18$.
    \item For $d=5$, the quantity $4|k_2| |\omega_2|^{n-d} - \frac{\mathcal{F}_{P_n}}{100}$ is negative for $n \ge 19$.
\end{itemize}
Thus, for $n \ge 19$ and $d \in \{3, 4, 5\}$, the bound holds.
\end{proof}

\begin{lemma} \label{lem:ratio6}
\[\frac{d-1}{\psi^{d}} \le \frac{1}{e \psi \ln \psi} < \frac{100}{101} < 1.\]
\end{lemma}

\begin{proof}
By continuous optimization, the function $f(d) = \frac{d-1}{\psi^d}$ attains its global maximum at exactly $d = 1 + \frac{1}{\ln \psi} \approx 4.56$. The maximum value is $f\left( 1 + \frac{1}{\ln \psi} \right) =  \frac{1}{e \psi \ln \psi} \approx 0.9876 < \frac{100}{101} \approx 0.9901$.
\end{proof}

\begin{lemma} \label{lem:ToC2n}
For $n \ge d \ge 3$
    \[ (d-1) \frac{{{n - d} \choose 2}}{\psi^d}  < \frac{100}{101} {{n-1} \choose 2}  \]
\end{lemma}

\begin{proof}
By applying Lemma \ref{lem:ratio6}, we have 
\[
(d-1) \frac{{{n - d} \choose 2}}{\psi^d} \le \frac{1}{e \psi \ln \psi} {{n - d} \choose 2} < \frac{100}{101} {{n-1} \choose 2}.
\]
\end{proof}

\begin{lemma} \label{lem:growth_53}
    For $n \ge 8$, the growth of $\mathcal{F}_{P_n}$ satisfies $\frac{5}{3} \left( \mathcal{F}_{P_n} -1 \right) \ge \binom{n-6}{2}$.
\end{lemma}

We omit a formal proof. $\frac{5}{3} \left( \mathcal{F}_{P_n} -1 \right)$ grows exponentially whereas $\binom{n-6}{2}$ grows quadratically. A plot of $\frac{5}{3} \left( \mathcal{F}_{P_n} -1 \right) - \binom{n-6}{2}$ is given in Figure \ref{fig:lem:growth_53}. \hfill $\square$

\begin{figure}[H]
\centering
    \includegraphics[width=0.5\textwidth]{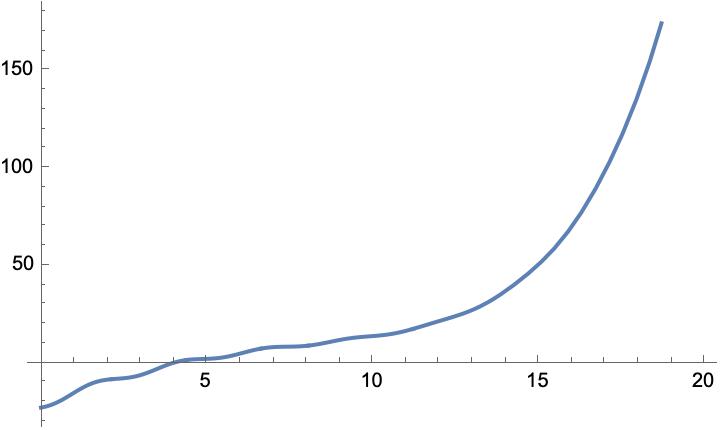}
    \caption{Plot of $\frac53 (\mathcal{F}_{P_n}-1) - \binom{n-6}{2}$ as referenced in Lemma \ref{lem:growth_53}.} \label{fig:lem:growth_53}
\end{figure}

\begin{lemma} \label{lem:dratio53}
For $d \ge 6$, the ratio $\frac{1}{d} \binom{d-1}{2} = \frac{(d-1)(d-2)}{2d}$ is non-decreasing and satisfies $\frac{1}{d} \binom{d-1}{2} \ge \frac{5}{3}$.
\end{lemma}

\begin{proof}
    The critical points of the function $f(d) = \frac{(d-1)(d-2)}{2d}$ are $d = -\sqrt{2},0, \sqrt{2}$ and the derivative, $f'(d) = -\frac{(d-1) (d-2)}{2 d^2} + \frac{d-2}{2 d} + \frac{d-1}{2 d}$ has $f'(6) = \frac{17}{36} > 0$. Hence, $f(d)$ is nondecreasing for $d\ge 6$. With $f(6) = \frac{5}{3}$, the result holds.
\end{proof}

\begin{lemma} \label{lem:19}
For $n \ge 8$ and $d \ge 6$, \[ \binom{d-1}{2} + d \binom{n-d}{2} \le \binom{d-1}{2} \mathcal{F}_{P_{n}}\].
\end{lemma}

\begin{proof}
Applying Lemma \ref{lem:growth_53} and Lemma \ref{lem:dratio53},

\[ \binom{n-d}{2} \le \binom{n-6}{2} \le \frac{5}{3} (\mathcal{F}_{P_{n}} - 1) \le \frac{1}{d} \binom{d-1}{2} (\mathcal{F}_{P_{n}} - 1). \]
Multiplying by $d$ and adding $\binom{d-1}{2}$ to both sides yields:
\begin{eqnarray*}
    \binom{d-1}{2} + d \binom{n-d}{2} &\le& \binom{d-1}{2} + \binom{d-1}{2} (\mathcal{F}_{P_{n}} - 1) \\ &=& \binom{d-1}{2} \mathcal{F}_{P_{n}}
\end{eqnarray*}
\end{proof}

\begin{lemma} \label{lem:dgreater6}
If $n \ge 8$ and $d \ge 6$, then 
\[ \binom{d-1}{2} + (d-1) \binom{n-d}{2} \mathcal{F}_{P_{n-d}} \le \binom{n}{2} \mathcal{F}_{P_{n}}. \]
\end{lemma}

\begin{proof}

\renewcommand{\arraystretch}{2.5}
\[ \begin{array}{>{\displaystyle}l >{\displaystyle}l >{\displaystyle}l >{\displaystyle}r}
& & \binom{d-1}{2} + ~(d-1)~ \binom{n-d}{2}~ \mathcal{F}_{P_{n-d}} & \\
& \le & \binom{d-1}{2} + ~(d-1)~ \binom{n-d}{2}~ \left( K \frac{\psi^{n}}{\psi^d} + 1 \right) & \text{\hspace{3cm} ($\mathcal{F}_{P_{n-d}} \le K \psi^{n-d} + 1$)} \\
& = & \binom{d-1}{2} + ~\binom{n-d}{2}~ K \psi^{n} \frac{d-1}{\psi^d} + ~(d-1)~ \binom{n-d}{2} & \\
& \le & \binom{d-1}{2} + ~\binom{n-d}{2}~ K \psi^{n} + ~(d-1)~ \binom{n-d}{2} & \text{\hspace{3cm} (Lemma \ref{lem:ratio6})} \\
& = & \binom{d-1}{2} + ~\binom{n-d}{2}~ K \psi^{n} - ~\binom{n-d}{2} + ~d~ \binom{n-d}{2} & \\
& = & \binom{d-1}{2} + ~\binom{n-d}{2}~ \left( K \psi^{n} - 1 \right) + ~d~ \binom{n-d}{2} & \\
& \le & \binom{d-1}{2} + ~\binom{n-d}{2}~ \mathcal{F}_{P_{n}} + ~d~ \binom{n-d}{2} & \text{\hspace{3cm} ($K \psi^{n} -1 \le \mathcal{F}_{P_{n}}$)} \\
& \le & \binom{d-1}{2}~ \mathcal{F}_{P_{n}} + ~\binom{n-d}{2}~ \mathcal{F}_{P_{n}} & \text{\hspace{3cm} (Lemma \ref{lem:19})} \\
& \le & \binom{n}{2}~ \mathcal{F}_{P_{n}}. & \\
\end{array} \]

\end{proof}

\subsection{Induction Step ($n \ge 20$)} \label{sec:mainresult}
For fixed $n \ge 20$, we apply strong induction. We assume the induction hypothesis:

\[ \text{{\bf Induction Hypothesis:} ~~~~~~~~~~~~~~ For any $k < n$ (except $k = 2$), } \mathcal{F}_{R_{k}} \le \binom{k}{2} \mathcal{F}_{P_{k}}. \]

By Lemma \ref{lem:recurr}, any forest $R$ on $n$ vertices satisfies at least one of two recursion inequalities given in the following two cases:

\paragraph{Case 1: Path-like Recursion.} 
Suppose $\mathcal{F}_{R_{n}} \le \mathcal{F}_{R_{n-2}} + \mathcal{F}_{R_{n-3}}$ as in inequality \ref{eqn:path}, then by the induction hypothesis:
\begin{equation*}
\mathcal{F}_{R_{n}} \le \binom{n-2}{2}\mathcal{F}_{P_{n-2}} + \binom{n-3}{2}\mathcal{F}_{P_{n-3}}.
\end{equation*}
Since $\binom{n-3}{2} \le \binom{n-2}{2} \le \binom{n}{2}$,
\begin{equation*}
\mathcal{F}_{R_{n}} \le \binom{n-2}{2}(\mathcal{F}_{P_{n-2}} + \mathcal{F}_{P_{n-3}}) = \binom{n-2}{2}\mathcal{F}_{P_{n}} \le \binom{n}{2}\mathcal{F}_{P_{n}}.
\end{equation*}

\paragraph{Case 2: Large Degree Case.}
Suppose instead there exists $d$ such that $3 \le d \le n-1$ ($d \ne n-2$) where $\mathcal{F}_{R_{n}} \le \binom{d-1}{2} + (d-1)\mathcal{F}_{R_{n-d}}$ as in inequality \ref{eqn:choices}. We distinguish between three subcases: $3 \le d \le 5$, $n-2 \ne d \ge 6$ and $d=n-2$.
~\\~\\
Subcase $3 \le d \le 5$:

\renewcommand{\arraystretch}{2.5}
\[ \begin{array}{>{\displaystyle}l >{\displaystyle}l >{\displaystyle}l >{\displaystyle}r}  
\mathcal{F}_{R_{n}} &\le& \binom{d-1}{2} + ~(d-1)~ \mathcal{F}_{R_{n-d}}  & \text{\hspace{3cm} (Given Inequality)} \\
&\le& \binom{d-1}{2} + ~(d-1)~ \binom{n-d}{2} ~\mathcal{F}_{P_{n-d}}  & \text{\hspace{3cm} (Induction Hypothesis)} \\
&\le& \binom{d-1}{2} + ~\frac{101}{100}~ (d-1)~ \binom{n-d}{2} ~\frac{\mathcal{F}_{P_{n}}}{\psi^d} & \text{\hspace{3cm} (Lemma \ref{lem:PndPnPsi}) } \\
&\le& \binom{d-1}{2} + ~\binom{n-d}{2} ~\mathcal{F}_{P_{n}} & \text{\hspace{3cm} (Lemma \ref{lem:ratio6}) } \\
&\le& \binom{d-1}{2} ~\mathcal{F}_{P_{n}} + ~\binom{n-d}{2} ~\mathcal{F}_{P_{n}} & \\
&\le& \binom{(d-1) + (n-d)}{2} ~\mathcal{F}_{P_{n}}&  \\
&=& \binom{n-1}{2} ~\mathcal{F}_{P_{n}} &\\
&\le& \binom{n}{2} ~\mathcal{F}_{P_{n}} &\\
\end{array} \]

Note that Lemma \ref{lem:PndPnPsi} only applies for $d= 3, 4, 5$, so a different approach must apply for $d \ge 6$.

Subcase $d \ge 6$ and $d \ne n-2$:


\renewcommand{\arraystretch}{2.5}
\[ \begin{array}{>{\displaystyle}l >{\displaystyle}l >{\displaystyle}l >{\displaystyle}r}
\mathcal{F}_{R_{n}} & \le & \binom{d-1}{2} + ~(d-1)~ \mathcal{F}_{R_{n-d}} & \text{\hspace{3cm} (Given Inequality)} \\
& \le & \binom{d-1}{2} + ~(d-1)~ \binom{n-d}{2} ~\mathcal{F}_{P_{n-d}} & \text{\hspace{3cm} (Induction Hypothesis)} \\
& \le & \binom{n}{2}~\mathcal{F}_{P_{n}} & \text{\hspace{3cm} (Lemma \ref{lem:dgreater6})} \\
\end{array} \]

If $d = n-2$, then the induction hypothesis does not apply. Even so, we have the following:

\[ \begin{array}{>{\displaystyle}l >{\displaystyle}l >{\displaystyle}l >{\displaystyle}r}
\mathcal{F}_{R_{n}} & \le & \binom{n-3}{2} + ~(n-3)~ \mathcal{F}_{R_{2}} & \text{\hspace{4cm} (Given Inequality)} \\
& = & \binom{n-3}{2} + ~(n-3)~ 2 & \text{\hspace{4cm} (Calculation for $\mathcal{F}_{R_2} = 2$)} \\
& \le & \binom{n-1}{2}  & \text{\hspace{4cm} (Convexity)} \\
& \le & \binom{n-1}{2}~\mathcal{F}_{P_{n}} 
\end{array} \]

This concludes the proof of Theorem \ref{thm:main}. \hfill 

\section{Conclusion and Open Questions} \label{sec:conclude}
\subsection{About the Main Theorem}

By using the principle of mathematical induction, we proved that $\mathcal{F}_{R_n}$ (and hence $\mathcal{F}_{T_n}$) cannot be larger than $\binom{n}{2} \mathcal{F}_{P_n} \sim k n^2 \psi^n$ for some constant $k$. We achieved this by establishing a large computer-assisted base case for all trees on $n \le 19$ vertices. This was necessary because many of the combinatorial inequalities only apply for larger $n$; namely, Lemma \ref{lem:PndPnPsi} does not apply until $n \ge 19$.

\subsection{Toward General Graphs}

In addition to their conjecture on trees, the authors of \cite{becker2025number} also ask about the growth rate for the number of minimal forts for general graphs. For trees, we have proven that the maximum growth rate is $\mathcal{F}_{T_n} = \psi^{n+o(n)} \approx  1.3247^{n+o(n)}$. The authors of \cite{becker2025number} give an example of a graph family that achieves $k \cdot (3^{1/3})^n \approx k\cdot 1.4422^n$. This leads to the question:

\begin{question}[From Conjecture 31 in \cite{becker2025number}]
Does there exist a family of graphs $\{G_n\}$ such that
\[
 \lim_{n \to \infty} \frac{\mathcal{F}_{G_{n+1}}}{\mathcal{F}_{G_n}} > 3^{1/3} ~?
\]

Or is $3^{1/3}$ best possible?

\end{question}

They frame the question as a conjecture in the affirmative.
During our study, we attempted to find such an example and could not. We believe that the example they give that achieves $3^{1/3}$ (multiple $K_4$ overlapping at a single vertex) may, in fact, be best possible. 

The techniques we use to bound $\mathcal{F}_{T_n}$ are often very specific for trees. Hence, establishing a bound for graphs in general will likely require a different approach.

For example, in Corollary \ref{lem:three_uncolored_in_a_row}, we show that any minimal fort for a tree cannot have three consecutive adjacent vertices. However, this is not true even in the case of unicyclic graphs. Such an example is given in Figure \ref{fig:unicyclicthreeinarow}. Also, general graphs may have minimal forts where non-fort vertices have three or more fort neighbors, in contrast to Corollary \ref{cor:treenot3neighbors}.

Hence, proving an optimal bound for the general graph case will require a far more sophisticated approach. 

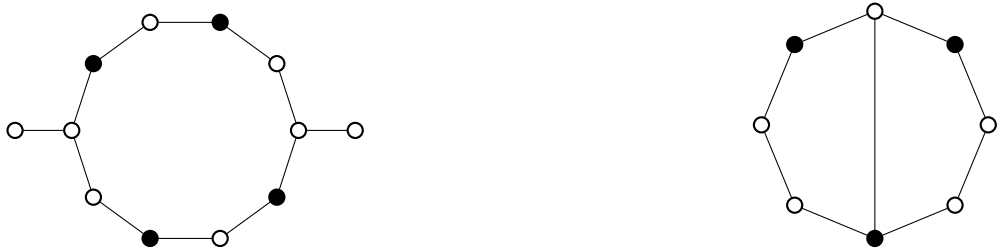
\begin{figure}[ht]
    \centering
    
    \begin{subfigure}[b]{0.45\textwidth}
        \centering
        \begin{tikzpicture}[scale=0.5, transform shape,
            stdnode/.style={circle, draw, minimum size=4mm, inner sep=0pt, thick},
            fillednode/.style={stdnode, fill=black},
            every edge/.style={draw, thick}
        ]
            \def\cycleradius{3cm}
            \def\leafradius{4.5cm}

            \foreach \i/\nodestyle in {
                1/stdnode, 2/fillednode, 3/stdnode, 4/fillednode, 5/stdnode,
                6/stdnode, 7/fillednode, 8/stdnode, 9/fillednode, 10/stdnode}
            {
                \node[\nodestyle] (c\i) at ({180 - (\i-1)*36}:\cycleradius) {};
            }
            \foreach \i [count=\nexti from 2] in {1,...,9} { \draw (c\i) -- (c\nexti); }
            \draw (c10) -- (c1);

            \node[stdnode] (l1) at (180:\leafradius) {};
            \draw (c1) -- (l1);
            \node[stdnode] (l2) at (0:\leafradius) {};
            \draw (c6) -- (l2);
        \end{tikzpicture}
    \end{subfigure}
    \hfill 
    \begin{subfigure}[b]{0.45\textwidth}
        \centering
        \begin{tikzpicture}[scale=0.5, transform shape,
            stdnode/.style={circle, draw, minimum size=4mm, inner sep=0pt, thick},
            fillednode/.style={stdnode, fill=black},
            every edge/.style={draw, thick}
        ]
            \def\cycleradius{3cm}

            \foreach \i/\nodestyle in {
                1/stdnode, 2/fillednode, 3/stdnode, 4/stdnode,
                5/fillednode, 6/stdnode, 7/stdnode, 8/fillednode}
            {
                \node[\nodestyle] (v\i) at ({90 - (\i-1)*45}:\cycleradius) {};
            }
            \foreach \i [count=\nexti from 2] in {1,...,7} { \draw (v\i) -- (v\nexti); }
            \draw (v8) -- (v1);

            \draw (v1) -- (v5);
        \end{tikzpicture}
    \end{subfigure}

    \caption{Examples of unicyclic graphs with forts (as uncolored vertices) that do not obey results for trees. (Left) An example of a graph with a fort with multiple instances of three consecutive adjacent vertices. This demonstrates that Corollary \ref{lem:three_uncolored_in_a_row} does not apply to general graphs. (Right) An example of a graph with a fort where one vertex is adjacent to three fort vertices. This demonstrates that Corollary  \ref{cor:treenot3neighbors} does not apply to general graphs.}
    
    \label{fig:unicyclicthreeinarow}
\end{figure}

\subsection{Mean Number of Minimal Forts}

In Section \ref{sec:compresults}, our computations briefly investigated the question of the {\it mean} number of minimal forts on a tree. 

\begin{question}
Let $\mathbb{E} \mathcal{F}_{T_n}$ denote the mean number of minimal forts over all nonisomorphic trees on $n$ vertices. What is the growth rate of $\mathbb{E} \mathcal{F}_{T_n}$? Must it be that $\mathbb{E} \mathcal{F}_{T_n} \ll \mathcal{F}_{T_n}$?
\end{question}

A naive exponential fit to the data for $\mathbb{E} \mathcal{F}_{T_n}$ in Table \ref{tab:comp} is $\mathbb{E} \mathcal{F}_{T_n} \approx 3.02188 \times 1.13438^n$, suggesting that $\mathbb{E} \mathcal{F}_{T_n}$ grows at a much smaller rate than $\mathcal{F}_{T_n} = k_1 \psi^n \approx k_1 \times 1.32472^n$.

\subsection{Identifying the True Extremal Tree for Larger $n$}

For $4 \le n \le 18$, our algorithm definitively concludes that the star graph, $S_n$, is the tree with the maximum number of minimal forts. At $n=19$, it shows a shift to the graph $T(19,4,4,2)$. However, we are not able to make the same determination for much larger $n$. The analysis shows that the path graph cannot be the extreme example until at least $n=73$. Even beyond that, it is not intuitive that the path graph must ``win out.''

\begin{question}
What tree(s) achieve $\mathcal{F}_{T_n}$ for $n \ge 20$? For $19 \le n \le 73$, is it $T(n,m,k,p)$? For sufficiently large $n$, does $\mathcal{F}_{T_n} = \mathcal{F}_{P_n}$?
\end{question}

\bibliographystyle{acm}
\bibliography{minimalforts}

\begin{thebibliography}{10}

\bibitem{aazami2008hardness}
{\sc Aazami, A.}
\newblock {\em Hardness results and approximation algorithms for some problems
  on graphs}.
\newblock PhD thesis, University of Waterloo Waterloo, Ontario, Canada, 2008.

\bibitem{abbas2022leader}
{\sc Abbas, W., Shabbir, M., Yaz{\i}c{\i}o{\u{g}}lu, Y., and Koutsoukos, X.}
\newblock Leader selection for strong structural controllability in networks
  using zero forcing sets.
\newblock In {\em 2022 American Control Conference (ACC)\/} (2022), IEEE,
  pp.~1444--1449.

\bibitem{AIM}
{\sc {AIM Minimum Rank--Special Graphs Work Group}}.
\newblock Zero forcing sets and the minimum rank of graphs.
\newblock {\em Linear algebra and its applications 428}, 7 (2008), 1628--1648.

\bibitem{becker2025number}
{\sc Becker, P., Cameron, T.~R., Hanely, D., Ong, B., and Previte, J.~P.}
\newblock On the number of minimal forts of a graph.
\newblock {\em Graphs and Combinatorics 41}, 1 (2025), 25.

\bibitem{brimkov2021improved}
{\sc Brimkov, B., Mikesell, D., and Hicks, I.~V.}
\newblock Improved computational approaches and heuristics for zero forcing.
\newblock {\em INFORMS Journal on Computing 33}, 4 (2021), 1384--1399.

\bibitem{brueni2005pmu}
{\sc Brueni, D.~J., and Heath, L.~S.}
\newblock The {{P}{M}{U}} placement problem.
\newblock {\em SIAM Journal on Discrete Mathematics 19}, 3 (2005), 744--761.

\bibitem{burgarth2013zero}
{\sc Burgarth, D., D'Alessandro, D., Hogben, L., Severini, S., and Young, M.}
\newblock Zero forcing, linear and quantum controllability for systems evolving
  on networks.
\newblock {\em IEEE Transactions on Automatic Control 58}, 9 (2013),
  2349--2354.

\bibitem{cameron2023forts}
{\sc Cameron, T.~R., Hogben, L., Kenter, F.~H., Mojallal, S.~A., and Schuerger,
  H.}
\newblock Forts,(fractional) zero forcing, and cartesian products of graphs.
\newblock {\em arXiv preprint arXiv:2310.17904\/} (2023).

\bibitem{cameron2025minimal}
{\sc Cameron, T.~R., and Li, K.}
\newblock On the minimal forts of trees.
\newblock {\em arXiv preprint arXiv:2512.12874\/} (2025).

\bibitem{fast2018effects}
{\sc Fast, C.~C., and Hicks, I.~V.}
\newblock Effects of vertex degrees on the zero-forcing number and propagation
  time of a graph.
\newblock {\em Discrete Applied Mathematics 250\/} (2018), 215--226.

\bibitem{hicks2024many}
{\sc Hicks, I.~V., and Brimkov, B.}
\newblock The many face(t)s of zero forcing.
\newblock {\em Notices of the American Mathematical Society 71}, 2 (2024),
  167--173.

\bibitem{liao2015hybrid}
{\sc Liao, C.-S., Hsieh, T.-J., Guo, X.-C., Liu, J.-H., and Chu, C.-C.}
\newblock Hybrid search for the optimal {{P}{M}{U}} placement problem on a
  power grid.
\newblock {\em European Journal of Operational Research 243}, 3 (2015),
  985--994.

\bibitem{McKayTrees}
{\sc McKay, B.}
\newblock {Combinatorial Data: Trees}.
\newblock \url{https://users.cecs.anu.edu.au/~bdm/data/}, 2025.
\newblock Accessed: 14 November 2025.

\bibitem{monshizadeh2014zero}
{\sc Monshizadeh, N., Zhang, S., and Camlibel, M.~K.}
\newblock Zero forcing sets and controllability of dynamical systems defined on
  graphs.
\newblock {\em IEEE Transactions on Automatic Control 59}, 9 (2014),
  2562--2567.

\bibitem{datkony_minimalforts}
{\sc {{N. H.} {{\DJ}ạt}}}.
\newblock Github: Minimal{F}orts.
\newblock \url{https://github.com/datkony/MinimalForts}, 2026.

\bibitem{trefois2015zero}
{\sc Trefois, M., and Delvenne, J.-C.}
\newblock Zero forcing number, constrained matchings and strong structural
  controllability.
\newblock {\em Linear Algebra and its Applications 484\/} (2015), 199--218.

\end{thebibliography}

\end{document}